\documentclass[12pt,a4]{amsart}
\usepackage{amsmath}
\usepackage{amssymb}
\usepackage{amsthm}
\usepackage{enumitem}
\usepackage{url}
\usepackage{soul}
\usepackage{times}
\usepackage[T1]{fontenc}

\usepackage{color}
\usepackage{dsfont}
\usepackage{epsfig}
\usepackage[hidelinks]{hyperref}
\usepackage{setspace}
\usepackage{epstopdf}
\usepackage[authoryear,round,sort]{natbib}
\usepackage{comment}
\usepackage{booktabs}
\usepackage{float}
\usepackage{tikz}
\usepackage{algorithm}
\usepackage{stmaryrd}
\numberwithin{equation}{section}

\usepackage[font=small,labelfont=bf]{caption}
\DeclareCaptionFont{tiny}{\tiny}
\usepackage{geometry}
\theoremstyle{plain}

\newtheorem{theorem}{Theorem}[section]
\newtheorem{lemma}{Lemma}[section]

\theoremstyle{definition}
\newtheorem{definition}{Definition}[section]
\newtheorem{assumption}{Assumption}[section]
\newtheorem{example}{Example}[section]

\theoremstyle{remark}
\newtheorem{rk}{Remark}[section]
\expandafter\let\expandafter\oldproof\csname\string\proof\endcsname
\let\oldendproof\endproof
\renewenvironment{proof}[1][\proofname]{%
  \oldproof[\noindent\textbf{#1.} ]%
}{\oldendproof}
\newcommand{\1}{\mathds{1}}

\newcommand{\E}{\mathbb{E}}

\newcommand{\be}{\begin{equation}}
\newcommand{\ee}{\end{equation}}
\newcommand{\by}{\begin{eqnarray*}}
\newcommand{\ey}{\end{eqnarray*}}

\renewcommand{\leq}{\leqslant}
\renewcommand{\geq}{\geqslant}
\usepackage{xcolor}
\definecolor{dark-red}{rgb}{0.4,0.15,0.15}
\definecolor{dark-blue}{rgb}{0.15,0.15,0.4}
\definecolor{medium-blue}{rgb}{0,0,0.5}
\hypersetup{
    colorlinks, linkcolor={red},
    citecolor={blue}, urlcolor={blue}
}
\allowdisplaybreaks

\begin{document}
\title{Velocity formulae between entropy and hitting time for Markov chains}
\author{Michael C.H. Choi}\thanks{}
\address{Institute for Data and Decision Analytics, The Chinese University of Hong Kong, Shenzhen, Guangdong, 518172, P.R. China}
\email{michaelchoi@cuhk.edu.cn}
\date{\today}
\maketitle


\begin{abstract}
	In the absence of acceleration, the velocity formula gives ``distance travelled equals speed multiplied by time''. For a broad class of Markov chains such as circulant Markov chains or random walk on complete graphs, we prove a probabilistic analogue of the velocity formula between entropy and hitting time, where distance is the entropy of the Markov trajectories from state $i$ to state $j$ in the sense of [L. Ekroot and T. M. Cover. The entropy of Markov trajectories. IEEE Trans. Inform. Theory 39(4): 1418-1421.], speed is the classical entropy rate of the chain, and the time variable is the expected hitting time between $i$ and $j$. This motivates us to define new entropic counterparts of various hitting time parameters such as average hitting time or commute time, and prove analogous velocity formulae and estimates between these quantities. 
	\smallskip
	
	
\end{abstract}



\section{Introduction and main results}

Suppose a particle moves from a point $i$ and to another point $j$. In elementary physics, the classical velocity formula yields the distance between $i$ and $j$ is equal to the speed of the particle multiplied by the time taken, provided that the particle has no acceleration. For the class of Markov chains with constant row entropy, the main aim of this note is to prove analogues of the velocity formula where ``distance'' is replaced by various entropic quantities, ``speed'' is the entropy rate associated with the chain and ``time'' is substituted by various hitting time related parameters such as average hitting time and commute time.

Before we discuss our main results in Theorem \ref{thm:veloform} and Theorem \ref{thm:mainresult2} below, we first fix our notations and provide a quick review on the relevant background. Our notations follow closely those of \cite{KGT13,CJ06,EC93}. Throughout this paper, we consider a discrete-time homogeneous irreducible finite Markov chain $X = (X_n)_{n \in \mathbb{N}}$ on state space $\mathcal{X}$ with transition matrix $P = (P_{i,j})_{i,j \in \mathcal{X}}$ and stationary distribution $\pi = (\pi_i)_{i \in \mathcal{X}}$. The entropy rate $H(X)$ of the Markov chain $X$ is defined to be
$$H(X) := - \sum_{i,j \in \mathcal{X}} \pi_i P_{i,j} \log P_{i,j} = \sum_{i \in \mathcal{X}} \pi_i H(P_{i,\cdot}),$$
where $H(P_{i,\cdot}) := - \sum_{j \in \mathcal{X}} P_{i,j} \log P_{i,j}$ is the one-step local entropy at state $i$, and the usual convention of $0 \log 0 = 0$ applies. $H(X)$ can be broadly interpreted as the average entropy produced by a single step of $X$, and this interpretation is particularly useful in understanding our main results. Another entropic quantity that we are interested in is the so-called entropy of the Markov trajectories $H_{i,j}$ from state $i$ to state $j$, as studied by \cite{EC93,KGT13}. Define a trajectory $T_{i,j}$ from $i$ to $j$ as a path with initial state $i$, final state $j$ with no intervening state equal to $j$. We denote such trajectory by $T_{i,j} = ix_1x_2\ldots x_{k-1}j$. The probability of $T_{i,j}$ is $p(T_{i,j}) := P_{i,x_1}P_{x_1,x_2}\ldots P_{x_{k-1},j}$. Writing $\mathcal{T}_{i,j}$ as the set of all possible trajectories from $i$ to $j$, $H_{i,j}$ is then defined to be
$$H_{i,j} = H_{i,j}(X) := -\sum_{T_{i,j} \in \mathcal{T}_{i,j}} p(T_{i,j}) \log p(T_{i,j}).$$

We now move on to discuss a few hitting time related parameters of $X$. Define $\tau_j := \inf\{n \geq 0;~X_n = j\}$ to be the first hitting time of the state $j$, and $\tau_j^+ := \inf\{n \geq 1;~X_n = j\}$ to be the first return time of the state $j$. The usual convention applies in these definitions with $\inf \emptyset = \infty$.

In our main results below, we primarily consider Markov chains with constant row entropy. In essence, this means that the Markov chain has zero entropic acceleration as it moves from one state to another since each state gives the same local entropy $H(P_{i,\cdot})$.

\begin{assumption}[Constant row entropy]\label{assump:constentr}
	We assume that $X$ has constant row entropy, i.e. for all $i,j \in \mathcal{X}$, $H(P_{i,\cdot}) = H(P_{j,\cdot})$.
\end{assumption}

Examples of such Markov chains can be found in Section \ref{sec:examples}, where we apply our results to two-state Markov chains (Example \ref{ex:twostate}), random walk on complete graphs (Example \ref{ex:complete}), rank-one Markov chains (Example \ref{ex:rankone}) and simple random walks on $n$-cycle (Example \ref{ex:ncycle}). Note that random walk on regular graphs and circulant Markov chains \cite{ACMM13} also fall into this category.

With the above notations and setting, we are now ready to state our main result. In a broad sense, it can be interpreted as the entropy of the trajectories from $i$ to $j$ equals the entropy per step times the mean hitting time between the two states.

\begin{theorem}[Velocity formula between entropy and hitting time]\label{thm:veloform}
	Assume that $X$ satisfies Assumption \ref{assump:constentr} with constant row entropy. For any $i,j \in \mathcal{X}$, we have
	\begin{align*}
	H_{i,j} =
	\begin{cases}
	\mathbb{E}_i(\tau_j) H(X), & \text{for } i \neq j,\\
	\mathbb{E}_i(\tau_i^+) H(X), & \text{for } i = j.
	\end{cases}
	\end{align*}
\end{theorem}

Note that for a deterministic Markov chain $X$, Theorem \ref{thm:veloform} trivially holds since $H_{i,j} = H(X) = 0$. Motivated by the relation between $H_{i,j}$ and $\E_i(\tau_j)$, we proceed to define a few new entropic parameters which are similar to their hitting time counterparts. We refer interested readers to \cite{LPW09,AF14,MT06} for excellent discussion on these parameters as well as their estimates.

\begin{definition}[Average entropy $H^{av}$, average hitting time $t^{av}$ and relaxation time $t^{rel}$]\label{def:aveavt}
	The average entropy and average hitting time are defined to be respectively 
	$$H^{av} = H^{av}(X) := \sum_{i,j \in \mathcal{X}} \pi_i \pi_j H_{i,j}, \quad t^{av} = t^{av}(X) := \sum_{i,j \in \mathcal{X}} \pi_i \pi_j \E_i(\tau_j).$$
	For reversible Markov chain $X$, a closely related parameter is the relaxation time 
	$$t^{rel} := \dfrac{1}{1-\lambda_2},$$
	where $1 = \lambda_1 > \lambda_2 \geq \ldots \geq \lambda_n$ are the eigenvalues of reversible $P$ arranged in non-increasing order and $n := |\mathcal{X}|$.
\end{definition}

\begin{definition}[Commute entropy $H^{c}_{i,j}$ and commute time $t^{c}_{i,j}$]\label{def:cect}
	For any $i,j \in \mathcal{X}$, the commute entropy and commute time between $i$ and $j$ are defined to be respectively
	$$H^{c}_{i,j} = H^{c}_{i,j}(X) := H_{i,j} + H_{j,i}, \quad t^{c}_{i,j} = t^{c}_{i,j}(X) := \E_i(\tau_j) + \E_j(\tau_i).$$
\end{definition}

We note that average entropy and average hitting time are both global parameters, while commute entropy and commute time are parameters associated with a given pair of states. In our second main result below, we give velocity formula between these parameters and carry a few results of hitting time to their entropic counterparts.

\begin{theorem}\label{thm:mainresult2}
	\begin{enumerate}
		\item(Commute entropy velocity formula)\label{it:come}
		For any $i \neq j \in \mathcal{X}$, we have
		$$H^c_{i,j} = t^c_{i,j} H(X).$$ 
		Note that this holds in general and does not require the constant row entropy assumption.
		
		\item(Average entropy velocity formula)\label{it:ave} Under the constant row entropy assumption \ref{assump:constentr}, we have
		$$H^{av} = (t^{av} + 1)H(X).$$
		If in addition $X$ is reversible, then
		$$(t^{rel} + 1) H(X) \leq H^{av} \leq ((|\mathcal{X}|-1)t^{rel} + 1) \log |\mathcal{X}|.$$
		
		\item(Entropic random target lemma)\label{it:rtl} Under the constant row entropy assumption \ref{assump:constentr}, 
		$$\sum_{j \in \mathcal{X}} \pi_j H_{i,j}$$
		does not depend on $i \in \mathcal{X}$.
		
		\item(Entropic cyclic tour lemma)\label{it:ctl} Under the constant row entropy assumption \ref{assump:constentr}, if $X$ is reversible then for any $i \neq j \neq k \in \mathcal{X}$
		$$H_{i,j} + H_{j,k} + H_{k,i} =  H_{i,k} + H_{k,j} + H_{j,i}.$$
	\end{enumerate}
\end{theorem}

\begin{rk}[Relation between eigenvalues and entropy]
	The relation between eigenvalues and entropy is perhaps best illustrated by item \eqref{it:ave}. For reversible Markov chains, using the so-called eigentime identity \cite{CuiMao10,AF14}, we have
	$$H^{av} = (t^{av} + 1)H(X) = \left(\sum_{i=2}^n \dfrac{1}{1-\lambda_i} + 1\right) H(X).$$
	Another important point to note is that the relaxation time bounds are tight. The lower bound is exactly attained by the two-state Markov chain, while the upper bound is attained by a rank-one Markov chain with row equals to the probability mass function of discrete uniform distribution. We refer readers to Section \ref{sec:examples} for further details.
\end{rk}

The rest of the paper is organized as follow. In Section \ref{sec:proofs} we give the proofs of Theorem \ref{thm:veloform} and \ref{thm:mainresult2}. In Section \ref{sec:examples} we provide a few examples to illustrate these two main results.

\section{Proofs of the main results}\label{sec:proofs}

\subsection{Proof of Theorem \ref{thm:veloform}}

We first state a lemma that relates the trajectory entropy $H_{i,j}$ to mean hitting times $\E_i(\tau_j)$ and the stationary distribution $\pi$. This result is the key to our proof and relies crucially on \cite{EC93}.

\begin{lemma}\label{lem: Hij}
	For $i,j \in \mathcal{X}$, we have
	\begin{align*}
		H_{i,j} = \begin{cases}
		\sum_{k \in \mathcal{X}} \pi_k (\E_j(\tau_k) - \E_i(\tau_k))H(P_{k,\cdot})+ \mathbb{E}_i(\tau_j) H(X), & \text{for } i \neq j,\\
		\mathbb{E}_i(\tau_i^+) H(X), & \text{for } i = j.
		\end{cases}
	\end{align*}
\end{lemma}

\begin{proof}
	The result in the case when $i = j$ can be directly obtained from existing results, since
	$$H_{i,i} = \dfrac{H(X)}{\pi_i} = \mathbb{E}_i(\tau_i^+) H(X),$$
	where the first equality follows from \cite[Theorem $1$]{EC93}, and the second equality is the well-known identity that the mean first return time of state $i$ is the inverse of $\pi_i$. 
	We proceed to consider the case when $i \neq j$. We first recall a result from \cite{EC93}, where $H_{i,j}$ can be formulated as, for $k \in \mathcal{X}$,
	\begin{align*}
		H_{i,j} &= K_{i,j} - K_{j,j}, \\
		K_{i,j} &= (ZB)_{i,j}, \\
		Z_{i,k} &= (I - P + \Pi)^{-1}_{i,k}, \\
		B_{k,j} &= H(P_{k,\cdot}) - \dfrac{H(X)}{\pi_k} \1_{k = j}, \\
		\Pi_{k,j} &= \pi_j,
	\end{align*}
	where $\1$ is the indicator function. Note that $Z = (Z_{i,k})_{i,k \in \mathcal{X}}$ is the fundamental matrix of the Markov chain $X$ and from \cite{AF14} it can be rewritten as
	\begin{align*}
		Z_{i,k} = \sum_{n = 0}^{\infty} P^n_{i,k} - \pi_k = Z_{k,k} - \pi_k \E_i(\tau_k).
	\end{align*}
	Plugging in this formula of $Z$ into $H$, we have
	\begin{align*}
		H_{i,j} &= \sum_{k \in \mathcal{X}} (Z_{i,k} - Z_{j,k}) B_{k,j} \\ 
		&= \sum_{k \in \mathcal{X}}  \pi_k (\E_j(\tau_k) - \E_i(\tau_k)) B_{k,j}  \\
		&= 	\sum_{k \in \mathcal{X}} \pi_k (\E_j(\tau_k) - \E_i(\tau_k))H(P_{k,\cdot})+ \mathbb{E}_i(\tau_j) H(X).
	\end{align*}
\end{proof}

With the above lemma in mind, we return to the proof of Theorem \ref{thm:veloform}, and it suffices for us to prove the case when $i \neq j$. Under the constant row entropy assumption \ref{assump:constentr}, by writing $H(P_{k,\cdot}) = H(P_{0,\cdot})$, the first term of $H_{i,j}$ in Lemma \ref{lem: Hij} can be simplified to
\begin{align*}
H_{i,j} &= \sum_{k \in \mathcal{X}} \pi_k (\E_j(\tau_k) - \E_i(\tau_k))H(P_{k,\cdot})+ \mathbb{E}_i(\tau_j) H(X) \\
&= H(P_{0,\cdot}) \sum_{k \in \mathcal{X}} \pi_k (\E_j(\tau_k) - \E_i(\tau_k)) + \E_i(\tau_j) H(X) = \E_i(\tau_j) H(X),
\end{align*}
where the third equality follows from the random target lemma, see for example \cite[Lemma $10.1$]{LPW09}.

\subsection{Proof of Theorem \ref{thm:mainresult2}}

We first prove item \eqref{it:come}. Note that
\begin{align*}
H^{c}_{i,j} &= H_{i,j} + H_{j,i} \\
			&= \sum_{k \in \mathcal{X}} \pi_k (\E_j(\tau_k) - \E_i(\tau_k))H(P_{k,\cdot})+ \mathbb{E}_i(\tau_j) H(X) + \sum_{k \in \mathcal{X}} \pi_k (\E_i(\tau_k) - \E_j(\tau_k))H(P_{k,\cdot})+ \mathbb{E}_j(\tau_i) H(X) \\
			&= t^c_{i,j} H(X),
\end{align*}
where the second equality follows from Lemma \ref{lem: Hij}. Next, we prove item \eqref{it:ave}, which follows directly from Theorem \ref{thm:veloform} since
\begin{align*}
H^{av} &= \sum_{i,j \in \mathcal{X}} \pi_i \pi_j H_{i,j} \\
	   &= \sum_{i \neq j} \pi_i \pi_j \E_i(\tau_j) H(X) + \sum_i \pi_i^2 H_{i,i} \\
	   &= t^{av} H(X) + \sum_i \pi_i H(X) = (t^{av}+1)H(X).
\end{align*}
If $X$ is in addition reversible, denote by $1 = \lambda_1 > \lambda_2 \geq \ldots \geq \lambda_n$ the eigenvalues of $P$ arranged in non-increasing order and $n := |\mathcal{X}|$. Using the eigentime identity \cite{AF14,CuiMao10} and elementary estimate gives
$$t^{rel} = \dfrac{1}{1-\lambda_2} \leq t_{av} = \sum_{i = 2}^n \dfrac{1}{1-\lambda_i} \leq (n-1)t^{rel}.$$
This together with $H(X) \leq \log |\mathcal{X}|$ yields the desired result. We proceed to prove item \eqref{it:rtl}. Using Theorem \ref{thm:veloform} again, we have
$$\sum_{j \in \mathcal{X}} \pi_j H_{i,j} = \sum_{j \neq i} \pi_j \E_i(\tau_j) H(X) + \pi_i \E_i(\tau_j^+) H(X) = \sum_{j \neq i} \pi_j \E_i(\tau_j) H(X) + H(X),$$
which does not depend on $i$ by random target lemma (see e.g. \cite[Lemma $10.1$]{LPW09}). Finally, we prove item \eqref{it:ctl}, which follows from Theorem \ref{thm:veloform} together with the cyclic tour lemma of mean hitting times for reversible $X$, see \cite[Lemma $10.10$]{LPW09}.

\section{Examples}\label{sec:examples}

In this section, we investigate in detail a few examples that illustrate Theorem \ref{thm:veloform} and \ref{thm:mainresult2}.

\begin{example}[Symmetric two-state Markov chains]\label{ex:twostate}
	In our first example, we consider a reversible two-state Markov chain on $\mathcal{X} = \{0,1\}$ with $P_{0,0} = P_{1,1} = 1 - p$ and $P_{0,1} = P_{1,0} = p$, where $p \in (0,1)$. The case for $p = 0$ or $1$ is trivial since $H(X) = H_{i,j} = 0$. Note that two-state Markov chains are frequently used in the study of finite Markov chains. For instance in \cite{DSC96} it is used for studying the log-Sobolev inequality. Coming back to our example, the constant row entropy assumption \ref{assump:constentr} is satisfied since $H(X) = H(P_{0,\cdot}) = H(P_{1,\cdot}) = - p \log p - (1-p) \log (1-p)$. It is easy to see that
	$$\E_0(\tau_1) = \E_1(\tau_0) = \dfrac{1}{p}, \quad \lambda_2 = 1-p,$$
	and so Theorem \ref{thm:veloform} and \ref{thm:mainresult2} now read
	\begin{align*}
		H_{0,1} &= H_{1,0} = \dfrac{1}{p} \left( - p \log p - (1-p) \log (1-p)\right), \\
		H^c_{0,1} &= \dfrac{2}{p} \left( - p \log p - (1-p) \log (1-p)\right), \\
		(t^{rel} + 1) H(X) &= H^{av} = \left(\dfrac{1}{p}+1\right) \left( - p \log p - (1-p) \log (1-p)\right) \leq (t^{rel} + 1) \log 2.
	\end{align*}
	That is, the lower bound of item \eqref{it:ave} in Theorem \ref{thm:mainresult2} is attained, while the upper bound is attained if $p = 1/2$.
\end{example}

\begin{example}[Random walk on complete graphs]\label{ex:complete}
	In the second example, we consider the reversible random walk on a complete graph without self-loop on $\mathcal{X} = \{0,1,\ldots,n-1\}$ and $n \in \mathbb{N}$. More precisely, for $i \neq j \in \mathcal{X}$, we take $P_{i,i} = 0$ and $P_{i,j} = 1/(n-1)$, and the stationary distribution is well-known to be a discrete uniform $\pi_i = 1/n$. According to \cite[Chapter $5$ Example $9$]{AF14}, the mean hitting times and eigenvalues of this random walk are
	$$\E_i(\tau_j) = n-1, \quad \lambda_2 = \lambda_3 = \ldots = \lambda_n = -1/(n-1).$$
	It also has constant row entropy with $H(X) = H(P_{0,\cdot}) = \log (n-1).$
	Plugging in these expressions into Theorem \ref{thm:veloform} and \ref{thm:mainresult2} yields, for $i \neq j \in \mathcal{X}$,
	\begin{align*}
	H_{i,j} &= (n-1)\log(n-1), \quad H^c_{i,j} = 2(n-1)\log(n-1), \\
	H^{av} &= \left(\dfrac{(n-1)^2}{n} + 1\right) \log (n-1), \\
	(t^{rel} + 1) H(X) &= \left(\dfrac{n-1}{n} + 1\right) \log (n-1) \leq  H^{av} \leq \left(\dfrac{(n-1)^2}{n} + 1\right) \log n = ((n-1)t^{rel} + 1) \log n.
	\end{align*}
	Note that the relaxation time upper bound on average entropy gives the correct order of $O(n \log n)$.
\end{example}

\begin{example}[Rank-one Markov chains]\label{ex:rankone}
	The main purpose of this example is to illustrate the upper bound of item \eqref{it:ave} in Theorem \ref{thm:mainresult2} can be exactly attained. Suppose we are given a discrete distribution with probability mass function $\pi = (\pi_i)_{i=0}^{n-1}$ on $\mathcal{X} = \{0,1,\ldots,n-1\}$ and $n \in \mathbb{N}$. For all $i,j \in \mathcal{X}$, we take $P_{i,j} = \pi_j$. As the transition matrix $P$ clearly has rank one, the eigenvalues are $\lambda_2 = \lambda_3 = \ldots = \lambda_n = 0$. The constant entropy assumption \ref{assump:constentr} is also satisfied with
	$H(X) = H(P_{0,\cdot}) = - \sum_{j \in \mathcal{X}} \pi_j \log \pi_j$. As a result, using the eigentime identity, we can compute the average entropy $H^{av}$ as
	$$H^{av} = - n \sum_{j \in \mathcal{X}} \pi_j \log \pi_j \leq n \log n = ((n-1)t^{rel} + 1) \log n.$$
	The upper bound is therefore attained if $\pi$ is the discrete uniform distribution. In other words, within the class of rank-one Markov chains, the average entropy is maximized when $\pi$ is discrete uniform.
\end{example}

\begin{example}[Simple random walks on $n$-cycle]\label{ex:ncycle}
	In the final example, we consider a simple random walk on $\mathcal{X} = \{0,1,\ldots,n-1\}$ and $n \in \mathbb{N}$. The transition matrix is given by $P_{i,j} = 1/2$ if $j = (i+1) \mod n$ or $j = (i-1) \mod n$ and $P_{i,j} = 0$ otherwise. This random walk has been studied in \cite{LPW09,AF14}, with eigenvalues given by $(\cos(2\pi j /n))_{j=0}^{n-1}$. The constant row entropy assumption is also satisfied with $H(X) = H(P_{0,\cdot}) = \log 2$, and by \cite[Chapter $5$ Example $7$]{AF14}
	$$\E_0(\tau_i) = i(n-i), \quad t^{av} = \dfrac{n^2-1}{6}.$$
	Theorem \ref{thm:veloform} and \ref{thm:mainresult2} now yield
	$$H_{0,i} = i(n-i) \log 2, \quad H^{av} = \left(\dfrac{n^2-1}{6} + 1\right) \log 2.$$
\end{example}

\section*{Acknowledgement}
The author would like to thank the anonymous referee for a careful reading of the manuscript.

\bibliographystyle{abbrvnat}
\bibliography{thesis}

\end{document}